\documentclass[12pt]{amsart}

\usepackage{amssymb}
\usepackage{bm}
\usepackage{graphicx}
\usepackage{psfrag}
\usepackage[usenames,dvipsnames]{xcolor}
\usepackage{enumerate}
\usepackage{multirow}
\usepackage{hyperref}

\usepackage{comment}

\usepackage{algpseudocode}

\usepackage{mathtools}

\usepackage{xy}
\input xy
\xyoption{all}

\newcommand{\excise}[1]{}

\newtheorem{thm}{Theorem}[section]
\newtheorem{lemma}[thm]{Lemma}

\newtheorem{cor}[thm]{Corollary}
\newtheorem{prop}[thm]{Proposition}

\newtheorem{question}[thm]{Question}

\newtheorem{goal}[thm]{Goal}

\theoremstyle{definition}

\newtheorem{example}[thm]{Example}
\newtheorem{remark}[thm]{Remark}
\newtheorem{defn}[thm]{Definition}

\newtheorem{notation}[thm]{Notation}

\numberwithin{equation}{section}

\newcommand\NN{\mathbb{N}}

\newcommand\RR{\mathbb{R}}
\newcommand\R{\mathbb{R}}

\newcommand\ZZ{\mathbb{Z}}
\newcommand\Z{\mathbb{Z}}

\newcommand\kk{\Bbbk}

\newcommand\uu{{\mathbf u}}
\newcommand\vv{{\mathbf v}}
\newcommand\ww{{\mathbf w}}
\newcommand\xx{{\mathbf x}}
\newcommand\yy{{\mathbf y}}
\newcommand\zz{{\mathbf z}}

\newcommand\cC{{\mathcal C}}

\DeclarePairedDelimiter{\norm}{\lVert}{\rVert}
\providecommand{\abs}[1]{\lvert#1\rvert}

\DeclareMathOperator\cone{cone} 
\DeclareMathOperator\adj{adj} 
\DeclareMathOperator{\dist}{dist}
\DeclareMathOperator{\interior}{interior}
 
\DeclareMathOperator\sop{s}

\newcommand\nsg{\mathcal{S}}
\newcommand\nsgtwo{\mathcal{T}}

\def\vec#1{\mathchoice{\mbox{\boldmath$\displaystyle\bf#1$}}
{\mbox{\boldmath$\textstyle\bf#1$}}
{\mbox{\boldmath$\scriptstyle\bf#1$}}
{\mbox{\boldmath$\scriptscriptstyle\bf#1$}}}

\begin{document}

\mbox{}

\title[Numerical semigroups via projections and via quotients]{Numerical semigroups via projections \\ and via quotients}

\author[Bogart]{Tristram Bogart}
\address{Departamento de Matem\'aticas \\ Universidad de los Andes \\ Bogot\'a, Colombia}
\email{tc.bogart22@uniandes.edu.co}

\author[O'Neill]{Christopher O'Neill}
\address{Mathematics Department\\San Diego State University\\San Diego, CA 92182}
\email{cdoneill@sdsu.edu}

\author[Woods]{Kevin Woods}
\address{Department of Mathematics\\Oberlin College\\Oberlin, OH 44074}
\email{kwoods@oberlin.edu}

\date{\today}

\begin{abstract}
We examine two natural operations to create numerical semigroups. We~say that a numerical semigroup $\nsg$ is $k$-normalescent if it is the projection of the set of integer points in a $k$-dimensional polyhedral cone, and we say that $\nsg$ is a $k$-quotient if it is the quotient of a numerical semigroup with $k$ generators. We prove that all $k$-quotients are $k$-normalescent, and although the converse is false in general, we prove that the projection of the set of integer points in a cone with $k$ extreme rays (possibly lying in a dimension smaller than $k$) is a $k$-quotient. The discrete geometric perspective of studying cones is useful for studying $k$-quotients:\ in particular, we use it to prove that the sum of a $k_1$-quotient and a $k_2$-quotient is a $(k_1+k_2)$-quotient. In addition, we prove several results about when a numerical semigroup is \emph{not} $k$-normalescent.
\end{abstract}

\subjclass[2020]{primary 20M14, 52B20; secondary 06F05, 13F65}

\keywords{numerical semigroup, normal affine semigroup, polyhedral cone}

\maketitle


\section{Introduction}
\label{sec:intro}

We denote $\NN = \{0,1,2,\dots\}$, and we define a \emph{numerical semigroup}\footnote{It is more standard to also require that $\gcd(\nsg)=1$, but we would like to prove our results in greater generality. In Proposition~\ref{prop:proj} and Remark~\ref{rem:equiv}, we will see that the two options are actually equivalent from our perspective.}
 to be a set $\nsg\subseteq\NN$ that is closed under addition and contains~0. A numerical semigroup can be defined by a set of generators,
\[\langle a_1,\ldots,a_n\rangle = \{a_1x_1+\cdots +a_nx_n:\ x_i\in\NN\},\]
and if $a_1,\ldots,a_n$ are the (unique) minimal set of generators of $\nsg$, we say that $\nsg$ has \emph{embedding dimension} $\mathsf e(\nsg) = n$.  For example,
\[\langle 3,5\rangle = \{0,3,5,6,8,9,10,\ldots\}\]
has embedding dimension 2.  

If $\nsg$ is a numerical semigroup, then an interesting way to create a new numerical semigroup 
is by taking the \emph{quotient}
\[
\frac{\nsg}{d} = \{ t \in \NN:\  dt \in \nsg\}
\]
by some positive integer $d$.  Note that $\nsg/d$ is itself a numerical semigroup, one that in particular satisfies $\nsg \subseteq \nsg/d \subseteq \NN$.  For example, 
\[
\frac{\langle 3,5\rangle}{2}=\{0,3,4,5,\ldots\}=\langle 3,4,5\rangle.
\]

The following definition was introduced in~\cite{BOW1}.  

\begin{defn}\label{def:quotientrank}
We say a numerical semigroup $\nsg$ is a \emph{$k$-quotient} if $\nsg=\langle a_1,\ldots,a_k\rangle/d$ for some positive integers $d, a_1, \ldots, a_k$.  The \emph{quotient rank} of $\nsg$ is the smallest $k$ such that $\nsg$ is a $k$-quotient, and we say $\nsg$ has \emph{full quotient rank} if its quotient rank is $\mathsf e(\nsg)$.  Note that $\langle a_1,\ldots, a_k\rangle = \langle a_1,\ldots, a_k\rangle/1$, so the quotient rank is always at most $\mathsf e(\nsg)$.
\end{defn}

Quotients of numerical semigroups appear throughout the literature over the past couple of decades~\cite{symmetriconeelement,symmetricquotient} as well as recently~\cite{harrisquotient,nsquotientgens}.  The well-studied family of proportionally modular numerical semigroups~\cite{fundamentalgaps} are known to be precisely those with quotient rank two~\cite{propmodular}.  See~\cite[Chapter~6]{numerical} for a thorough overview of quotients.  In~\cite{BOW1}, we gave a sufficient condition for a numerical semigroup to have full quotient rank, as well as explicit examples with arbitrarily large quotient rank, and showed that ``almost all'' numerical semigroups have full quotient rank.

Seemingly unrelated to the above, 
\emph{normal affine semigroups} are subsets of $\ZZ^k$ of the form $C \cap \ZZ^k$, where $C\subseteq\RR^k$ is a pointed rational polyhedral cone (with vertex at the origin), that is,
\[
C
= \cone(\vec v_1,\ldots, \vec v_\ell)
= \{ \lambda_1 \vec v_1 + \cdots + \lambda_\ell \vec v_\ell:\ \lambda_i \in \RR_{\ge 0} \},
\]
for some $\vec v_1,\ldots, \vec v_\ell \in \ZZ^k$.  (See \cite[Chapters 7 and 8]{Schrijver} for background on cones.)
Although the only (nonnegative) one-dimensional normal affine semigroup is $\langle 1\rangle=\NN$, we can obtain other numerical semigroups as the image of higher dimensional normal affine semigroups under a projection, since linear maps preserve additive closure.

\begin{defn}\label{def:normrank}
We say a numerical semigroup $\nsg$ is \emph{$k$-normalescent}\footnote{The word \emph{normalescent} is meant to evoke that it is obtained from a normal semigroup via the process of projection; similar variants on the word ``normal'' tend to have some established meaning.}
if $\nsg = \pi(\cC \cap \ZZ^m)$, where $\cC \subseteq \RR^m$ is a $k$-dimensional rational polyhedral cone and $\pi:\RR^m\rightarrow \RR$ is a linear map with integer coefficients.  The \emph{normalescence rank} of $\nsg$ is the smallest~$k$ such that $\nsg$ is $k$-normalescent, and we say $\nsg$ has \emph{full normalescence rank} if its normalescence rank is $\mathsf e(\nsg)$.
Note $\langle a_1, \ldots, a_k \rangle = \pi(\cC \cap \ZZ^k)$, where $\cC = \RR_{\ge 0}^k$ and 
\[
\pi(x_1,\ldots,x_k) = a_1x_1 + \cdots + a_kx_k,
\]
so the normalescence rank of $\nsg$ is at most $\mathsf e(\nsg)$.
\end{defn}

It is convenient to allow $\cC$ and $\pi$ to have negative coordinates, though we must have $\pi(\cC) \subseteq \RR_{\ge 0}$ or else $ \pi(\cC \cap \ZZ^m)$ would contain negative integers.

\begin{example}
\emph{Arithmetical} numerical semigroups, which have the form 
$$\nsg = \langle a, a + h,\ldots, a + nh \rangle,$$
are 2-normalescent. Indeed, choose $C = \cone\big((1,0),(1,n)\big)$ and $\pi(x,y) = ax + hy$. For~$0\le i\le n$, $(1,i)\in C$ has $\pi(1,i)=a+ih$. Since $\{(1,0), \dots, (1,n)\}$ is easily seen to generate the normal affine semigroup $C \cap \ZZ^2$, this yields $\nsg = \pi(C \cap \ZZ^2)$.  These semigroups are known to have quotient rank~two~\cite{propmodular}, identical to their normalescence rank.  
\end{example}

The classification of $k$-normalescent semigroups is an interesting question for several reasons.  
On one hand, in the study of toric varieties~\cite{cls}, $\pi$ can be thought of as inducing a positive grading on the normal semigroup algebra $R = \kk[\cC \cap \ZZ^k]$ over a field~$\kk$, so that $\pi(\cC \cap \ZZ^k)$ equals the set of $\pi$-graded degrees of monomials in $R$.  In this setting, our question becomes:\ ``which numerical semigroups arise as the set of degrees of a normal semigroup algebra?''  
On the other hand, from the viewpoint of semigroup theory, we will note an intriguing, but easily proven, connection:\ all $k$-quotients are $k$-normalescent.  We will see that the converse is false in general (Proposition~\ref{p:counterexample}), but our main result is to prove a partial converse that is already quite powerful.  

\begin{defn}\label{def:simpnorm}
The \emph{extreme rays} of a cone $\cC\subseteq \R^m$ are the minimal set of $\vec v_1,\ldots,\vec v_k$ such that $\cC=\cone(\vec v_1,\ldots,\vec v_k)$.
A numerical semigroup $\nsg$ is \emph{$k$-ray-normalescent} if $\nsg = \pi(\cC \cap \ZZ^m)$, where $\cC \subseteq \RR^m$ is a rational polyhedral cone with $k$ extreme rays and $\pi:\RR^m \rightarrow \RR$ is a linear map with integer coefficients. 
\end{defn}

  
\begin{thm}[Main Theorem]\label{thm:main}
A numerical semigroup is a $k$-quotient if and only if it is $k$-ray-normalescent.
\end{thm}

Note that a $k$-dimensional cone $\cC$ must have \emph{at least} $k$ extreme rays. If $\cC$ has exactly $k$ extreme rays, then it is called $\emph{simplicial}$, and $\pi(\cC \cap \NN^m)$ will be both $k$-normalescent and $k$-ray-normalescent. If $\cC$ has $\ell>k$ extreme rays, then $\pi(\cC \cap \NN^m)$ will be $\ell$-ray-normalescent --- and hence an $\ell$-quotient --- but it might not be a $k$-quotient. Indeed, the proof of the following proposition uses a cone in $\RR^3$ with four extreme rays, so its projection $\nsg$ will be 3-normalescent and a 4-quotient, but $\nsg$ is not a 3-quotient.  

\begin{prop}\label{p:counterexample}
The numerical semigroup $\nsg = \langle 101, 102, 110, 111 \rangle$ is 3-normalescent but not a 3-quotient.  
\end{prop}

The above example has minimal dimension, in the sense that any 2-normalescent numerical semigroup is a 2-quotient (this follows from the fact that every 2-dimensional cone is simplicial).  This was observed in~\cite{proportionallymodularfullsemigroups}, using the fact that 2-quotients are precisely the family of proportionally modular numerical semigroups.  

The ability to translate between $k$-quotients and $k$-ray-normalescent semigroups is powerful, especially because it allows one to utilize tools from polyhedral geometry to prove things about $k$-quotients. For example, supposing $\nsg$ is a $k_1$-quotient and $\nsgtwo$ is a $k_2$-quotient, must $\nsg+\nsgtwo$ be a $(k_1+k_2)$-quotient?  This is not at all obvious, and we were unable to obtain a direct semigroup-theoretical proof in~\cite{BOW1}.  But the corresponding statement for normalescence (and ray-normalescence) is fairly easy to prove; we do so here to illustrate the power of the discrete geometry perspective.

\begin{thm}\label{thm:normsum}
If numerical semigroups $\nsg_1$ and $\nsg_2$ are $k_1$-(ray-)normalescent and $k_2$-(ray-)normalescent, respectively, then the numerical semigroup $\nsg_1+\nsg_2$ is $(k_1+k_2)$-(ray-)normalescent.
In particular, if $\nsg_1$ is a $k_1$-quotient and $\nsg_2$ is a $k_2$-quotient, then $\nsg_1+\nsg_2$ is a $(k_1+k_2)$-quotient. 
\end{thm}

\begin{proof}
For each $i = 1, 2$, let $\cC_i \subseteq \RR^{m_i}$ be a rational cone and $\pi_i$ be a projection such that $\nsg_i = \pi_i(\cC_i \cap \ZZ^{m_i})$.  
Let 
\[
\cC = \big\{\lambda_1(\vec x_1, \vec 0)+\lambda_2(\vec 0,\vec x_2):\ \vec x_i\in \cC_i,\, \lambda_i\ge 0\big\}\subseteq\RR^{m_1+m_2}\]
and $\pi(\vec x_1, \vec x_2) = \pi_1(\vec x_1) + \pi_2(\vec x_2)$ for $\vec x_i \in \R^{m_i}$.  
Notice any
\[
\lambda_1 (\vec x_1, \vec 0) + \lambda_2 (\vec 0, \vec x_2) \in \cC \cap \ZZ^{m_1+m_2}
\]
necessitates $\lambda_i \vec x_i \in \cC_i \cap \ZZ^{m_i}$, so
\[
\pi(\cC \cap \ZZ^{m_1+m_2})
= \pi_1(\cC_1 \cap \ZZ^{m_1}) + \pi_2(\cC_2 \cap \ZZ^{m_2})
= \nsg_1 + \nsg_2.
\]
The proof is complete upon observing that $\dim \cC=\dim \cC_1+\dim \cC_2$, giving us additivity of normalescence, and that each extreme ray of $\cC$ comes from an extreme ray of $\cC_1$ or of $\cC_2$, giving us additivity of ray-normalescence.
\end{proof}


\begin{example} \label{ex:bigquotient}
The final claim of Theorem~\ref{thm:normsum} (additivity of quotient rank) was proven in \cite[Theorem 2.3]{BOW1} with the additional hypothesis that the quotient denominators are coprime, in which case
\[\frac{\nsg}{c}+\frac{\nsgtwo}{d} = \frac{d\nsg+c\nsgtwo}{cd}.\]
While one can thus easily write
\[
\langle 23,24,25,29,30,31,32 \rangle
= \frac{\langle 23,25\rangle}{2}+\frac{\langle 29,32\rangle}{3} = \frac{3\langle 23,25\rangle + 2\langle 29,32\rangle}{2 \cdot 3}
= \frac{\langle 58,64,69,75\rangle}{6}
\]
as a 4-quotient, we could not find such a ``nice'' representation of
\[
\langle 23,24,25,29,30,31 \rangle
= \frac{\langle 23,25\rangle}{2} + \frac{\langle 29,31\rangle}{2}
\]
as a $4$-quotient (we were able to check by exhaustive search that is is not a 4-quotient with denominator ten or less). Using the tools developed in this paper, we can show that it is the 4-quotient
\[
\frac{\langle 13775465, 14996610, 18887728, 20196837  \rangle}{109340422}.
\]
These large values appear difficult to avoid in general when the denominators have a common factor.  Our proof of Theorem~\ref{thm:main} highlights the broader toolset the polyhedral geometric perspective brings to the table when studying numerical semigroup quotients. For example, it relies on the careful perturbation of the extreme rays of the cone and analysis of the expected Smith Normal Form of a large, random matrix (see \cite{Stanley2}).  
\end{example}

The paper is organized as follows.

In Section~\ref{sec:outline}, we develop our intuition about $k$-normalescence, see some examples, and outline the proof of Theorem~\ref{thm:main}. This includes a complete proof of the easier direction, that all $k$-quotients are $k$-ray-normalescent (Proposition~\ref{prop:easy}).

In Section~\ref{sec:ranks}, we prove Proposition~\ref{p:counterexample} and along the way develop a necessary condition for $k$-normalescence (Corollary~\ref{cor:necessary}).  This allows us to extend several results of~\cite{BOW1} about quotient rank to results about normalescence rank. In particular, we give explicit examples of numerical semigroups with arbitrarily large normalescense rank (Theorem~\ref{thm:noquotient}), as well as prove that ``almost all'' numerical semigroups have full normalescence rank (Theorem~\ref{thm:numericalbox}).

In Section~\ref{sec:aux}, we prove two propositions from Section~\ref{sec:outline} that require careful use of Smith Normal Form (see Definition~\ref{def:snf}), and in Section~\ref{sec:converse}, we use the ideas we have developed plus some more polyhedral geometry to prove the remaining (harder) implication of Theorem~\ref{thm:main}:\ that all $k$-ray-normalescent semigroups are $k$-quotients.

Mathematica~\cite{mathematica} code for many of the algorithms in this paper, including creating examples like Example~\ref{ex:bigquotient}, may be found on GitHub \cite{githubK}.

We close this section with one of our primary lingering questions.  

\begin{question}\label{q:decidable}
Is there an algorithm that computes normalescence rank?  How about quotient rank?
\end{question}

We conjecture that the answer is yes, but at the time of writing, it is not even known if these questions are decidable for $k \ge 3$ (the case $k = 2$ is addressed in~\cite{diophantineinequality}).  

\section{Outline of Main Proof}
\label{sec:outline}

We begin by stating a useful simplification of the problem (proved in Section~\ref{sec:aux}):\ in our equation $\nsg = \pi(\cC \cap \NN^k)$, we may assume that $\cC$ is full-dimensional and that $\pi$ is the projection onto the first coordinate.  



\begin{prop}\label{prop:proj}
If $\nsg$ is $k$-normalescent with $\gcd(\nsg) = d$, then there exists a full-dimensional cone $\cC \subseteq \RR^k$ such that $\nsg = \pi(\cC \cap \ZZ^k)$, where $\pi(\vec x) = dx_1$ is given by projection onto a multiple of the first coordinate.
Furthermore, if $\nsg$  is $k$-ray-normalescent, then we may take $\cC$ to be simplicial (i.e., generated by $k$ linearly independent rays).
\end{prop}

\begin{remark}\label{rem:equiv}
Note that if $\gcd(\nsg)=d>1$, then Proposition~\ref{prop:proj} shows that, by taking the projection $\pi/d$, we may instead examine the semigroup obtained by dividing every element of $\nsg$ by $d$. In particular, $d\nsg$ is $k$-(ray)-normalescent if and only if $\nsg$ is. One can verify from definitions (see Remark 1.3 of \cite{BOW1}) that $d\nsg$ is a $k$-quotient if and only if $\nsg$ is. Therefore, without loss of generality, we may assume $\gcd(\nsg)=1$.
\end{remark}

\begin{notation}\label{not:projmatrix}
Unless otherwise stated, from now on any numerical semigroup $\nsg$ will be assumed to have $\gcd(\nsg) = 1$.  
Given a rank $k$ matrix $M\in \ZZ^{k\times \ell}$ with columns $\vec v_1, \ldots, \vec v_\ell \in \ZZ^k$, define 
\[
\sop(M)=\pi(\cC \cap \ZZ^k), 
\qquad \text{where} \qquad
\cC = \cone(M)=\cone(\vec v_1,\ldots,\vec v_\ell)
\]
and $\pi$ is the projection onto the first coordinate.  
\end{notation}

Using Notation \ref{not:projmatrix}, we can rephrase Proposition \ref{prop:proj} as follows.

\begin{cor} \label{cor:proj}
A numerical semigroup  $\nsg$ (with $\gcd(\nsg) = 1$) is $k$-normalescent if and only if there exist $\ell \ge k$ and $M \in \ZZ^{k\times \ell}$ such that $\sop(M)=\nsg$. Furthermore, $\nsg$ is $k$-ray-normalescent if we may take $\ell = k$ so that $M$ is a square matrix.
\end{cor}

\begin{example}\label{ex:small}
We have noted that $\langle 11,13\rangle$ is 2-normalescent via the cone $\R_{\ge 0}^2$ and projection $(x,y)\mapsto 11x+13y$, but it is also 2-normalescent via the cone generated by $(11,5)$ and $(13,6)$ and projection $(x,y)\mapsto x$: this is a consequence of 
\[\det \begin{bmatrix} 11 & 13\\ 5 & 6\end{bmatrix}=\pm 1,\]
as we shall discuss in the next example.  Using Notation~\ref{not:projmatrix}, we write
\[\langle 11,13\rangle = \sop\left(\begin{bmatrix} 11 & 13\\ 5 & 6\end{bmatrix}\right).\]
\end{example}

\begin{example}\label{ex:unimodular}
Suppose that $M\in\Z^{k\times k}$ is a \emph{unimodular} matrix, that is, it has determinant $\pm 1$ and so is invertible over $\ZZ$. In this case, the corresponding cone $\cC=\cone(M)$ is also called \emph{unimodular}. If $\xx\in\cC\cap \Z^k$, then
\[\xx = M\left(M^{-1}\xx\right)\]
is a nonnegative \emph{integer} combination of the columns of $M$.  
In particular, if $[a_1 \cdots a_k]$ is the first row of $M$, then
\[\sop(M)=\langle a_1,\ldots,a_k\rangle.\]
\end{example}

With the above reduction in hand, we readily prove the easier half of Theorem~\ref{thm:main}. We do this via the following fact that will be used again in Section \ref{sec:converse}.

\begin{lemma} \label{lem:timesD}
Let $M$ be any $k \times \ell$ integer matrix and let $D$ be the $k \times k$ diagonal matrix $\textup{diag}(1,d,d,\dots,d)$. Then $\sop(M) / d  = \sop(DM)$.  
\end{lemma}

\begin{proof}
The product $DM$ multiplies every row of $M$ by $d$ except the first.  Thus,
\begin{align*}
t\in \sop(M) / d & \Leftrightarrow dt\in \sop(M)\\
&\Leftrightarrow \exists \vec x\in\Z^{\ell -1}:\ (dt,\vec x)\in \cone(M)\\
&\Leftrightarrow \exists \vec x\in\Z^{\ell -1}:\ (t,\vec x/d)\in \cone(M)\\
&\Leftrightarrow \exists \vec x\in\Z^{\ell -1}:\ (t,\vec x)\in \cone(DM)\\
&\Leftrightarrow t\in \sop(DM),
\end{align*}  
which implies $\sop(M) / d  = \sop(DM)$.  
\end{proof}

\begin{prop}\label{prop:easy}
All $k$-quotients are $k$-ray-normalescent.
\end{prop}

\begin{proof}
Let a $k$-quotient $\nsg=\langle a_1, \ldots, a_k \rangle/d$ be given. First note that $\langle a_1,\ldots,a_k\rangle$ equals the image of $\RR_{\ge 0}^k \cap \ZZ^k$ under the projection $(x_1,\ldots,x_k) \mapsto a_1x_1 + \cdots + a_kx_k$, and thus is itself $k$-ray-normalescent.  
By Corollary~\ref{cor:proj}, there is some $M \in \Z^{k\times k}$ such that $\langle a_1,\ldots,a_k\rangle = \sop(M)$.  
Now, letting $D$ be the $k\times k$ diagonal matrix with diagonal $(1,d,\ldots,d)$, Lemma~\ref{lem:timesD} implies $\nsg = \sop(DM)$, so $\nsg$ is $k$-ray-normalescent.  
\end{proof}

\begin{example}\label{ex:small2}
Continuing Example~\ref{ex:small}, we have
\[\frac{\langle 11,13\rangle}{2} = \sop\left(\begin{bmatrix}1 & 0\\ 0 & 2\end{bmatrix}\cdot\begin{bmatrix} 11 & 13\\ 5 & 6\end{bmatrix}\right)=\sop\left(\begin{bmatrix}11 & 13\\10& 12\end{bmatrix}\right).\]
Note that $(12,11)=\frac{1}{2}(11,10)+\frac{1}{2}(13,12)$ is in the cone, and similarly $12\in \langle 11,13\rangle/2$.
\end{example}

Now we outline the proof of the converse, that $k$-ray-normalescent implies $k$-quotient, with the full proof relegated to Section~\ref{sec:converse}. We are given a $k\times k$ full rank matrix $M$, and we want to detect whether $\sop(M)$ can be written as a $k$-quotient. If we are lucky, the \emph{Smith Normal Form}~\cite{Smith} of $M$ has a special property, which will immediately imply that $\sop(M)$ is a $k$-quotient.

\begin{defn}\label{def:snf}
Given a matrix $M\in\Z^{m \times \ell}$, a \emph{Smith Normal Form} for $M$ is a factorization $M=UDV$ such that:

\begin{itemize}
\item $D$ is a (rectangular) diagonal matrix $D\in\Z^{m \times \ell}$,
\item the main diagonal $(d_1,\ldots,d_n)$ of $D$ (where $n=\min(\ell,m)$) consists of nonnegative integers $d_i$ such that $d_i$ divides $d_{i+1}$ for all $i$,
\item  $U \in \Z^{m \times m}$ and $V \in \Z^{\ell \times \ell}$ are unimodular matrices (that is, they have determinant $\pm 1$ and so are invertible over the integers).
\end{itemize}
\end{defn}

See~\cite{Newman} for a broad overview. In particular, every integer matrix may be put in Smith Normal Form, and the diagonal $(d_1,\ldots,d_n)$ is unique (so often we simply call $(d_1,\ldots,d_n)$ the Smith Normal Form, or SNF, of $M$). The matrices $U$ and $V$ need not be unique. Furthermore, each product $d_1 d_2 \cdots d_i$ equals the gcd of the $i \times i$ minors of $M$ (with the convention that $\gcd(0,0) = 0$). The Smith Normal Form of an integer matrix is a useful tool in discrete geometry and the theory of integer lattices.  See~\cite{Stanley1} for an introduction with applications to combinatorics.

The following condition shows what SNF we need in order to guarantee we have a $k$-quotient (we save the proof for Section~\ref{sec:aux}).  

\begin{thm}\label{thm:SNF}
Let $M \in \ZZ^{k\times k}$ be a full-rank matrix with positive first row $(a_1, \ldots, a_k)$.  If $a_1, \ldots, a_k$ are relatively prime and the SNF for $M$ is $(1,d,d,\ldots,d)$ for some positive integer $d$, then
\[
\sop(M)= \frac{\langle a_1, \ldots, a_k\rangle}{d},
\]
and so $\sop(M)$ is a $k$-quotient.
\end{thm}

\begin{example}
Continuing Example~\ref{ex:small2}, 
\[M=\begin{bmatrix}11 & 13\\10& 12\end{bmatrix}=\begin{bmatrix}1& 0\\ 0 & 1\end{bmatrix}\cdot\begin{bmatrix}1 & 0\\ 0 & 2\end{bmatrix}\cdot\begin{bmatrix} 11 & 13\\ 5 & 6\end{bmatrix}\]
is in Smith Normal Form with diagonal $(1,2)$, and so we immediately recover
\[\sop(M)=\frac{\langle 11,13\rangle}{2} .\]
\end{example}

\begin{example}
\label{ex:reproof}
More generally, let $a_1,\ldots,a_k$ be relatively prime positive integers, and let $\vec a$ be the $1\times k$ matrix $[a_1\,\cdots\,a_k]$. Since the gcd of the $1 \times 1$ minors of $\vec a$ is $1$, the Smith Normal Form for $\vec a$ is $(1)$.  This means there exists a unimodular matrix~$V$ with
\begin{equation}
\label{eqn:transform}
\vec a \cdot V=[1]\cdot \vec a \cdot V = [1 \, 0 \, \cdots \, 0].
\end{equation}
Letting $M=V^{-1}$, we see that $[1 \, 0 \, \cdots \, 0]\cdot M=\vec a$, that is, the first row of $M$ is $\vec a$, and so $s(M)=\langle a_1,\ldots,a_k\rangle$, by Example~\ref{ex:unimodular}.

Now let $d$ be given, and let $D$ be the $k\times k$ diagonal matrix with diagonal $(1,d,\ldots,d)$. Then $IDM$ is in Smith Normal Form, has first row $\vec a$, and meets the criteria of Theorem~\ref{thm:SNF}. Therefore
\[\sop(IDM)=\frac{\langle a_1,\ldots,a_k\rangle}{d},\]
and we have another way of seeing that all $k$-quotients are $k$-ray-normalescent.
\end{example}

%

Unfortunately, Theorem~\ref{thm:SNF}, requires us to be lucky:\ only if $M$ is of the required form will it immediately guarantee that $\sop(M)$ is a $k$-quotient. If we are not lucky, then the next key idea is to try perturbing $M$ slightly. The following heuristic indicates that there are plenty of perturbed matrices $M'$ that maintain $\sop(M')=\sop(M)$:  since we may assume that $\gcd(\nsg)=1$, we know that every sufficiently large positive integer is in $\nsg$; therefore, if we make $\cC'=\cone(M')$ just slightly larger than $\cC=\cone(M)$, $\cC'\setminus \cC$ will certainly contain integer points, but we should be able to ensure that their first coordinates are large enough to already be in $\nsg$, and therefore $\sop(M')=\sop(M)$. Hopefully by examining enough such perturbed $M'$ we can find one of them that meets the hypotheses of Theorem~\ref{thm:SNF}, which will prove that $\sop(M)=\sop(M')$ is a $k$-quotient.

\begin{example}\label{ex:idea}
Take
\[
M = \begin{bmatrix} 6 & 8 \\ 1 & 1\end{bmatrix}.
\]
We see $\sop(M)=\langle 6,7,8\rangle$, by checking that $\{(6,1),(7,1),(8,1)\}$ generates $\cone(M)\cap\Z^2$ as a normal affine semigroup. We want to prove that $\sop(M)$ is a 2-quotient.  The first row of $M$ is not relatively prime, so we cannot use Theorem~\ref{thm:SNF}. But let's look at a new matrix
\[
M' = \begin{bmatrix} 6 & 25 \\ 1 & 3\end{bmatrix}.
\]
The cone $\cC'=\cone(M')$ is slightly larger than the cone $\cC=\cone(M)$; for example $(25,3) \in \cC' \setminus \cC$. But we can check that, after projecting onto the first coordinate, we still have $\sop(M') = \sop(M)$; for example, $(25,3)$ projects to $25$, but we already had that $25 = 3\cdot 6 +1\cdot 7\in \langle 6,7,8\rangle = \sop(M)$. This new matrix $M'$ has first row relatively prime and SNF $(1,7)$, and so
\[
\langle 6,7,8\rangle=\sop(M) = \sop(M') = \frac{\langle 6,25\rangle}{7}
\]
is a 2-quotient.
\end{example}

Example \ref{ex:idea} involves a $2 \times 2$ matrix $M$ violating the hypothesis of Theorem~\ref{thm:SNF} that $M$'s first row must be relatively prime. However, when $k > 2$ the hypothesis that the SNF of $M$ is $(1,d,\ldots,d)$ turns out to be even more restrictive. Indeed, Wang and Stanley showed \cite{Stanley2} that most random integer matrices will have SNF $(1,\ldots,1,d)$, which is almost the ``opposite'' of what we want. This indicates that we will rarely be lucky enough to be able to apply Theorem~\ref{thm:SNF}.


In order to get around this problem, we use our one last trick. Recall that the adjugate of a full rank matrix $B\in\ZZ^{k\times k}$ is the integer matrix
$\adj(B)=\det(B)B^{-1}$. 

\begin{lemma}\label{lem:adj}
If $B \in \ZZ^{k \times k}$ with SNF $(1,\ldots,1,d)$, then $\adj(B)$ has SNF $(1,d,\ldots,d)$.  
\end{lemma}

\begin{proof}
Say we have $B=UDV$ in SNF, where $U$ and $V$ are unimodular matrices and $D$ is the diagonal matrix with diagonal $(1,\ldots,1,d)$.  Let $A=\adj(B)$. Then
\[A=\det(B)B^{-1}=dB^{-1}=dV^{-1}D^{-1}U^{-1}=V^{-1}\left(dD^{-1}\right)U^{-1}.\]
We see that $dD^{-1}$ is a diagonal matrix with diagonal $(d,\ldots,d,1)$. Therefore, after switching the first and last rows/columns with elementary operations, we see that $A=\adj(B)$ has SNF $(1,d,\ldots,d)$. 
\end{proof}

So our final step is this:\ let $M'$ be a matrix such that any small integer perturbation of $M'$ will still project to the numerical semigroup $\sop(M)$. Let $B$ be a matrix that is a slight integer perturbation of $\adj(M')$. Since $\adj(\adj(M'))$ is a multiple of $M'$, it generates the same cone, that is,
\[\sop(\adj(\adj(M')))=\sop(M')=\sop(M).\]
We will see that $A=\adj(B)$ is a small perturbation of this multiple of $M'$, so $\sop(A)=\sop(M)$. If $B$ has SNF $(1,\ldots,1,d)$, then Lemma~\ref{lem:adj} implies that $A$ will have SNF $(1,d,\ldots,d)$, and we will be nearly done! Fortunately,  \cite{Stanley2} indicates that such $B$ should be easy to find, and \cite{EW} provides the exact result that we need. We leave the details to Section~\ref{sec:converse}. This process is how we obtained the 4-quotient
\[ \frac{\langle 23,25\rangle}{2} + \frac{\langle 29,31\rangle}{2} = \frac{\langle 13775465, 14996610, 18887728, 20196837  \rangle}{109340422}\]
in Example~\ref{ex:bigquotient}. The process of finding an $M'$ that allows ``wiggle room'' for perturbation and then taking adjugates twice contributes to the explosion in magnitudes of the generators and denominator.

\section{Quotient Rank and Normalescence Rank}
\label{sec:ranks}

Theorem~2.1 of our first paper~\cite{BOW1} identifies a necessary condition for a given numerical semigroup to be a $k$-quotient.  This condition was the principal ingredient in several subsequent results in~\cite{BOW1}, including (for any given $k\ge 3$) the first known example of a numerical semigroup that is not a $k$-quotient. We now prove that the same condition is also necessary for $k$-normalesence.

\begin{prop}\label{p:necessary}
Suppose $\nsg$ is $k$-normalescent.  Given any elements $s_1,\ldots,s_p \in \nsg$ with $p > k$, there exists a nonempty subset $I\subseteq \{1,\ldots,p\}$ such that $\tfrac{1}{2}\sum_{i\in I}s_i \in \nsg$.
\end{prop}

\begin{proof}
Suppose $\nsg = \sop(M)$ for some full rank integer matrix $M$ with $k$ rows, define $\cC=\cone(M) \subseteq \RR^k$, and fix $\vec b_1, \ldots, \vec b_p \in \cC \cap \ZZ^k$ so that $s_i$ is the first coordinate of~$\vec b_i$.  For a vector $\vec v \in \ZZ^k$, define $\vec v\bmod 2 \in \ZZ_2^k$ to be the coordinate-wise reduction of $\vec v$ modulo 2.  For each subset $J \subseteq \{1,\ldots,p\}$, we define $\vec b_J=\sum_{j\in J}\vec b_j$, and consider the reduction $\vec b_J \bmod 2$.  There are $2^p$ possible sets $J$ and $2^k$ possible values for $\vec b_J \bmod 2$, with $p>k$, so by the pigeonhole principle there must be distinct sets $J_1$ and $J_2$ with
\[
\vec b_{J_1} \bmod 2 = \vec b_{J_2} \bmod 2.
\]
Let $I=(J_1\setminus J_2)\cup (J_2\setminus J_1)$,
which is nonempty.
Then
\[
\vec b_I \bmod 2 = \vec b_{J_1} + \vec b_{J_2} - 2\vec b_{J_1\cap J_2} \bmod 2 = \vec 0,\]
so $\tfrac{1}{2}\vec b_I$ is an integer vector.
Therefore $\tfrac{1}{2}\vec b_I \in \cC \cap \ZZ^k$, so
\[
\frac{1}{2}
\sum_{i\in I} s_i \in \sop(M) = \nsg,
\]
as desired.  
\end{proof}

We record here three subsequent results whose proofs are identical to those of Corollary~2.2 and Theorems~3.1 and~4.1 of~\cite{BOW1}, respectively, now that Proposition~\ref{p:necessary} has been obtained.  

\begin{cor}\label{cor:necessary}
Let $\nsg = \langle a_1, \dots, a_n \rangle$ be a numerical semigroup.  If $\nsg$ does not have full normalescence rank, then there is a nonempty $I \subseteq \{1,\ldots,n\}$ such that
\[
\sum_{i\in I}a_i \in \langle a_j : j\notin I \rangle.
\]
\end{cor}

\begin{thm}\label{thm:noquotient}
Given positive integers $k$ and $a \ge 2^k$, 
the numerical semigroup
\[
\nsg = \langle 2a + 2^i : 0 \le i \le k \rangle
\]
has full normalescence rank $k+1$. 
\end{thm}

\begin{thm}\label{thm:numericalbox}
Fix $n \in\ZZ_+$, and let $q\in\ZZ_+$ vary.  If $a_1, \ldots, a_n \in \{1,\ldots,q\}$ are uniformly and independently chosen, then the probability that $\nsg = \langle a_1, \dots, a_n \rangle$ has full normalescence rank tends to 1 as $q \to \infty$.  More precisely, this probability is $1 - O(q^{-\frac{1}{n}})$.  
\end{thm}

At the time of writing, Proposition~\ref{p:necessary} and Corollary~\ref{cor:necessary} are the only known necessary conditions for $k$-normalescence, and their analogous results in~\cite{BOW1} are the only known necessary conditions for $k$-quotientability.  In particular, these conditions fail to distinguish $k$-normalescent semigroups from $k$-quotients. We now prove Proposition~\ref{p:counterexample}, which shows that they are indeed distinct concepts.

\begin{proof}[Proof of Proposition~\ref{p:counterexample}]
We first show that $\nsg$ is 3-normalescent. Let
\[
\uu_1 = (101, 1, 0), \quad
\uu_2 = (102, 1, 0), \quad
\uu_3 = (110, 0, 1), \quad
\text{and} \quad
\uu_4 = (111, 0, 1),
\]
let $M$ be the $3\times 4$ matrix with columns $\uu_1, \uu_2, \uu_3, \uu_4$, and let $\cC = \cone(M) \subseteq \RR^3$. Then $\sop(M)$ contains the generators of $\nsg$ and therefore contains $\nsg$. On the other hand, let $M_1$ have columns $\uu_1, \uu_3, \uu_4$ and $M_2$ have columns $\uu_1, \uu_2, \uu_4$. Both of these matrices are unimodular, and so (as discussed in Example~\ref{ex:unimodular}), we have
\[
\sop(M_1)=\langle 101, 110, 111\rangle
\quad\text{and}\quad
\sop(M_2)=\langle 101, 102, 111\rangle.
\]
But $\cone(M_1)\cup\cone(M_2)=\cC$, so in fact
\[\sop(M)=\sop(M_1)\cup \sop(M_2)=\nsg.\]

We now show that $\nsg$ is not 3-ray-normalescent, which will complete the proof by Proposition~\ref{prop:easy}.  Suppose by way of contradiction that $\nsg = \sop(M')$ for some $M'\in\Z^{3\times 3}$.  Let $\ww_1$, $\ww_2$, $\ww_3$, and $\ww_4$ be lattice points in $\cC'=\cone(M')$ whose respective first coordinates are 101, 102, 110, and 111, so that they project to the minimal generators of $\nsg$.
Applying Corollary~\ref{cor:necessary} to the generators of $\nsg$, we conclude that there exists a nonempty set $I \subseteq \{1,\ldots,4\}$ such that $\sum_{i\in I}\ww_i \in \langle \ww_j: \, j\notin I \rangle$ (the statement of Corollary~\ref{cor:necessary} only concerns projections of the $\ww_j$, but one can readily observe in the proof of Proposition~\ref{p:necessary} that the claimed expression descends from one involving vectors).  By considering the first coordinates of the four points, we can easily check that the only possibility is that $\ww_1 + \ww_4 = \ww_2 + \ww_3$.

Consider the lattice points
$$\begin{array}{r@{}c@{}l@{\qquad}r@{}c@{}l}
\vv_{12} &{}={}& 2 \ww_2 + \ww_1 - \ww_4, &
\vv_{24} &{}={}& 2 \ww_2 + \ww_4 - \ww_1, \\
\vv_{13} &{}={}& 2 \ww_3 + \ww_1 - \ww_4, &
\vv_{34} &{}={}& 2 \ww_3 + \ww_4 - \ww_1.
\end{array}$$
Their respective first coordinates are 194, 210, 214, and 230, none of which lie in $\nsg$, so none of them belong to $\cC'$. By the pigeonhole principle, one of the three inequalities that define $\cC'$ must be violated by at least two of the four points.

However, note that
\[
\vv_{12} + \vv_{13} = (2\ww_2 + 2\ww_3) + (2\ww_1 - 2\ww_4) = (2\ww_1 + 2\ww_4) + (2\ww_1 - 2\ww_4) = 4 \ww_1 \in \cC',
\]
so the segment between $\vv_{12}$ and $\vv_{13}$ passes through $\cC'$, and thus it cannot be the case that a single hyperplane separates both $\vv_{12}$ and $\vv_{13}$ from $\cC'$.  Similarly, the sums
\[
\vv_{12} + \vv_{24} = 4 \ww_2, 
\quad
\vv_{13} + \vv_{34} = 4 \ww_3,
\quad \text{and} \quad
\vv_{24} + \vv_{34} = 4 \ww_4
\]
all belong to $\cC'$, so none of these pairs of points can be separated from $\cC'$ by the same hyperplane.  Finally, 
\[
\vv_{12} + \vv_{34} = \vv_{13} + \vv_{24} = 2(\ww_2 + \ww_3) \in \cC'
\]
so neither the pair $\{\vv_{12}, \vv_{34}\}$ nor the pair $\{\vv_{13},\vv_{24}\}$ can be separated from $\cC'$ by a single hyperplane.  This is a contradiction.  
\end{proof}

\section{Proofs of Proposition~\ref{prop:proj} and Theorem~\ref{thm:SNF}}
\label{sec:aux}

Here we prove Proposition~\ref{prop:proj} and Theorem~\ref{thm:SNF}. 

\begin{proof}[Proof of Proposition~\ref{prop:proj}]
This is a somewhat technical proof that gets the desired outcome over multiple steps. The key idea is that Smith Normal Form is a useful tool for transforming $\RR^m$ in a way that respects the integer lattice, e.g., by transforming a non-full-dimensional cone into a full-dimensional cone in lower ambient dimension.

Suppose $\nsg$ is $k$-normalescent, so there is a $k$-dimensional cone $\cC\subseteq\RR^m$ (for some $m\ge k$) and a projection $\pi:\RR^m \rightarrow \RR$ such that $\nsg = \pi(\cC \cap \ZZ^m)$. Suppose that $\cC$ has $\ell$~extreme rays, and let $M \in \ZZ^{m \times l}$ be the matrix whose columns are the extreme rays of $\cC$. This means that any point in $\cC$ can be written as $M\xx$ with $\xx\in\RR_{\ge 0}^\ell$ (since $M\xx$ is a nonnegative real combination of the columns of $M$). Identify $\pi$ with its corresponding $1 \times m$ row vector, so that $\pi(\vec y)=\pi\cdot\vec y$ when we think of $\vec y$ as a column vector.

Our goal is to replace $\cC$ by a full-dimensional cone in $\RR^k$ and $\pi$ by the projection $\RR^k\rightarrow \RR$ onto (a multiple of) the first coordinate. Using the Smith Normal Form $M = UDV$, we will apply a series of modifications to achieve the intended goal. 
\smallskip

\noindent\textbf{Step 1:} We first absorb the unimodular matrix $U$ into the projection. To do this, observe that 
\[ \nsg = \left\{ \pi M \xx: \, \xx \in \RR_{\ge 0}^\ell,\, M \xx \in \ZZ^m \right\} = \left\{ \pi UDV \xx: \, \xx \in \RR_{\ge 0}^\ell,\, UDV \xx \in \ZZ^m \right\}.\]
Let $\pi_1 = \pi U \in \ZZ^{1 \times m}$ and $M_1 = DV$.  Since the unimodular matrix $U$ represents a bijection from $\ZZ^m$ to itself, we then have
\[ \nsg =  \left\{ \pi_1 M_1 \xx: \, \xx \in \RR_{\ge 0}^\ell,\, UM_1 \xx \in \ZZ^m \right\}=\left\{ \pi_1 M_1 \xx: \, \xx \in \RR_{\ge 0}^\ell,\, M_1 \xx \in \ZZ^m \right\}.\]

\noindent\textbf{Step 2:} Since $\cC$ is $k$-dimensional, $M_1$ is of rank $k$. We next replace the matrix $M_1$ with a $k\times \ell$ matrix of full rank, thus making the cone full-dimensional. Since $M_1=DV$ has rank $k$, the last $m-k$ rows of $D$ are zero. Let $D_2$ be the matrix consisting of the first $k$ rows of $D$ and $M_2 = D_2V \in \ZZ^{k \times \ell}$, which is also of rank $k$ because $V$ is invertible. Let $\pi_2$ be the $1 \times k$ row vector consisting of the first $k$ entries of $\pi_1$. Then    
\[
M_1 \xx 
= \left[ \begin{array}{c}  D_2 \\ \hline 0 \end{array} \right] V \xx
= \left[ \begin{array}{c}  D_2V \xx \\ \hline 0 \end{array} \right]
= \left[ \begin{array}{c}  M_2 \xx \\ \hline 0 \end{array} \right]
\]
so $M_1 \xx \in \ZZ^m$ if and only if $M_2 \xx \in \ZZ^k$. Furthermore,
\[
\pi_1 M_1 \xx
= \pi_1 \left[ \begin{array}{c}  M_2 \xx \\ \hline 0 \end{array} \right]
= \pi_2 M_2 \xx,
\]
and so
\[
\nsg =  \left\{ \pi_2 M_2 \xx: \, \xx \in \RR_{\ge 0}^\ell,\, M_2 \xx \in \ZZ^k \right\}.
\]

\noindent\textbf{Step 3:}
We now want to change our projection to be projection onto a multiple of the first coordinate. Let $d=\gcd(\nsg)$, and let $\cC_2=\cone(M_2)$. Every element of $\nsg$ will be a multiple of $\gcd(\pi_2)$, so certainly $\gcd(\pi_2)$ divides $d$. We want to show that $\gcd(\pi_2)=d$.

Since $\cC_2$ is full-dimensional, there exists a rational point $y''$ in the (topological) interior of $\cC_2$. Scaling by the common denominator of the coordinates of $y''$, we obtain an integer point $y' \in \interior(\cC_2)$. The distance $\delta$ from $y'$ to the boundary of $\cC_2$ is strictly positive; let $y$ be an integer point obtained by scaling $y'$ by any integer $N > 1/\delta$.  The distance from $y$ to the boundary of $\cC_2$ equals $N \delta > 1$, so in particular, $\vec y + \vec e_i \in \cC_2$ for each $i = 1, \ldots,k$.  
As such, 
\[
(\pi_2)_i = \big(\pi_2(\vec y + \vec e_i) - \pi_2(\vec y)\big)
\]
is a difference of two integers in $\nsg$, and therefore $d$ divides $(\pi_2)_i$, for all $i$. 

Therefore $\gcd(\pi_2)=d$.  Similarly to~\eqref{eqn:transform} in Example~\ref{ex:reproof}, the Smith Normal Form for $\pi_2$ is $(d)$, and so there exists a unimodular matrix~$W$ such that
\[\pi_2W=[1]\cdot\pi_2 \cdot W = [d \, 0 \, \cdots \, 0].\] 

Let $\pi_3=\pi_2W=[d \, 0 \, \cdots \, 0]$ and $M_3=W^{-1}M_2$. Then
\begin{align*}
\nsg&=\left\{ \pi_2 M_2 \xx: \, \xx \in \RR_{\ge 0}^\ell,\, M_2 \xx \in \ZZ^k \right\}\\
&=\left\{ \pi_2WW^{-1}M_2 \xx: \, \xx \in \RR_{\ge 0}^\ell,\, M_2 \xx \in \ZZ^k \right\}\\
&=\left\{ \pi_3 M_3 \xx: \, \xx \in \RR_{\ge 0}^\ell,\, M_3 \xx \in \ZZ^k \right\},
\end{align*}
using that $W^{-1}$ is a bijection of the integer lattice. Note that $\pi_3$ is the desired projection onto $d$ times the first coordinate.
\medskip

This completes the proof of Proposition~\ref{prop:proj} when $\nsg$ is $k$-normalescent.  Now we turn to ray-normalescence. Notice that the $\nsg$ that we have been examining is actually $\ell$-ray-normalescent, for some $\ell\ge k$. If $\ell>k$, we actually want to create an $\ell$-dimensional cone in $\RR^\ell$ projecting to $\nsg$, so we need to increase the ambient dimension from $k$ to $\ell$, while also increasing the dimension of the cone from $k$ to $\ell$ at the same time. We insert the following between Steps 2 and 3:
\medskip

\noindent\textbf{Step 2.5:}
Recall that our matrix is $M_2=D_2V$, where $D_2$ is a $k\times \ell$ diagonal matrix of full rank and $V$ is an $\ell\times\ell$ unimodular matrix. We want to ``lift'' each extreme ray in $\cone(M_2)$ (these are the columns of $M_2$) up into an $\ell$-dimensional space, by adding $\ell-k$ new rows to $M_2$ to make a new matrix $M'_2$, and we want to do it in such a way that $\cone(M'_2)$ is $\ell$-dimensional; our new projection $\pi'_2$ will simply forget about these last $\ell-k$ rows.

In particular, let $(d_1, \dots, d_k)$ be the diagonal entries of $D_2$, which are all nonzero. Let $t$ be the least common multiple of the nonzero maximal minors of $M_2$; there is at least one nonzero maximal minor, since $M_2$ is of full rank. Let $D'_2$ be the $\ell \times \ell$ diagonal matrix whose diagonal entries are $(d_1, \dots, d_k, t, \dots, t)$. Now $D'_2$ and $V$ are both $\ell \times \ell$ matrices of full rank, so the matrix $M'_2 = D'_2 V$ is as well, and thus $\cone(M'_2)$ is  a full-dimensional cone in $\RR^\ell$. 
Let $\pi'_2$ be the $1 \times \ell$ matrix obtained by appending $\ell-k$ zeros to $\pi_2$. 

We want to show that 
\[\left\{ \pi'_2 M'_2 \xx: \, \xx \in \RR_{\ge 0}^\ell,\, M'_2 \xx \in \ZZ^\ell \right\}=\left\{ \pi_2 M_2 \xx: \, \xx \in \RR_{\ge 0}^\ell,\, M_2 \xx \in \ZZ^k \right\}.\]
The forward inclusion, $\subseteq$, is clear: $M_2$ is the first $k$ rows of $M'_2$, so $M'_2\xx\in\Z^\ell$ implies $M_2\xx\in\Z^k$; and $\pi_2M_2=\pi'_2M'_2$, so $\pi_2M_2\xx=\pi'_2M'_2\xx$ (that is, an integer point in $\cone(M'_2)$ projects to an integer point in $\cone(M_2)$ when we simply forget about the last $\ell - k$ coordinates, so it will ultimately project to a point in $\nsg$).

For the reverse inclusion, let $\xx \in \RR^\ell_{\geq 0}$ be such that $M_2 \xx \in \ZZ^k$. If $M'_2\xx\in\ZZ^\ell$, we would be done, because $\pi'_2M'_2\xx = \pi_2M_2\xx$, but this need not be the case. Instead,  we must find a $\yy\in\RR^\ell_{\geq 0}$ such that
\[M'_2\yy\in\ZZ^\ell \quad\text{and}\quad M_2\yy=M_2\xx\]
(the second equation says that $M_2\yy$ and $M_2\xx$ are two different ways of writing the same point as a nonnegative linear combination of the extreme rays of $\cone(M_2)$), which will imply that 
\[\pi'_2M'_2\yy= \pi_2M_2\yy=\pi_2M_2\xx=\pi'_2M'_2\xx\]
and complete the reverse inclusion.

To find such a $\yy$, we apply Carath\'eodory's theorem \cite{DGK} (see \cite[Corollary 7.1a]{Schrijver} for the exact form we are using): since $M_2\xx\in\cone(M_2)$, there exist $k$ linearly independent columns of $M_2$ such that $M_2\xx$ is in the cone they generate; that is, there exist a $k\times k$ nonsingular submatrix $Q$ of $M_2$ and $\zz\in\RR^k_{\geq 0}$ such that $Q\zz=M_2\xx$. Let $\yy\in\RR^\ell_{\geq 0}$ be  identical to $\zz$ on the entries corresponding to the columns of $M_2$ that comprise $Q$ and 0 on the remaining entries, so that $M_2\yy=Q\zz=M_2\xx.$ Then
 all that remains to prove is that $M'_2\yy\in\ZZ^\ell$.

Indeed, we know $Q\zz=M_2\xx\in\ZZ^k$, and multiplying on the left by the adjugate matrix $\adj(Q)=\det(Q)Q^{-1}$ (which has integer entries), we get that $(\det Q) \zz \in \ZZ^k$. Therefore $t \zz \in \ZZ^k$ because $\det(Q)$ is one of the nonzero maximal minors whose lcm is $t$. Furthermore, $t \yy \in \ZZ^k$ since the entries of $\yy$ that are not also entries of $\zz$ are just zeros, and in fact $t (\vv \cdot \yy) = \vv \cdot (t \yy) \in \ZZ$ for any integer vector $\vv$. In particular, we conclude that
\[
M'_2 \yy
= \left[ \begin{array}{c@{}c@{}c}  & D_2 &  \\ \hline 0 & \vline & tI \end{array} \right] V \yy
= \begin{bmatrix} M_2 \yy \\ t \vv_{k+1} \cdot \yy \\ \vdots \\ t \vv_{\ell} \cdot \yy \end{bmatrix} \in \ZZ^\ell,\]
where $\vv_1, \dots, \vv_\ell$ are the rows of $V$.
\end{proof}

\begin{proof}[Proof of Theorem~\ref{thm:SNF}]
We will transform $M$ into Smith Normal Form, but we want to do it carefully by finding unimodular matrices $U$ and $V$ of a particular form so that $UMV=D$ (recall $U$ and $V$ are not unique). We will also use the letters $U$ and $V$ for the intermediate matrices as we compute the SNF:\ that is, at the beginning we have $U=V=I$, and $UMV=M$, and at the end we will have $UMV=D$, where $D$ has diagonal $(1,d,\ldots,d)$.

As in Example~\ref{ex:reproof}, Equation~\ref{eqn:transform}, we first let $V$ be the unimodular matrix such that
\[[a_1\,\cdots\,a_k]\cdot V=[1\,0\,\cdots\,0],\]
that is, the first row of $MV$ is  $[1\,0\,\cdots\,0]$. Noting that $U$ corresponds to elementary row operations, we can now subtract multiples of the first row from the other rows so that the first column is  $[1\,0\,\cdots\,0]^T$, that is
\[
UMV = \begin{bmatrix} 1 & \vec 0 \\ \vec 0 & M' \end{bmatrix}
\]
in block form, where $M'$ is a $(k-1)\times (k-1)$ matrix. Since these row operations did not alter the first row, the first row of $U$ is  $[1\,0\,\cdots\,0]$. Now put $M'$ in SNF using elementary row and column operations, and we will end with $D=UMV$ and the first row of $U$ is still $ [1\,0\,\cdots\,0]$. Note that $D$ is indeed the diagonal matrix with diagonal $(1,d,\ldots,d)$ by the uniqueness of the SNF diagonal. The first row of $UM$ will be $[a_1\,\cdots\,a_k]$, the first row of $M$. Since $V^{-1}=D^{-1}UM$ and $D^{-1}, U$ have first row $[1\,0\,\cdots\,0]$, the first row of $V^{-1}$ will also be $[a_1\,\cdots\,a_k]$.

In summary, we have found a unimodular matrix $V^{-1}$ whose first row is $[a_1\,\cdots\,a_k]$, so $\sop(V^{-1})=\langle a_1,\ldots,a_k\rangle$ (see Example~\ref{ex:unimodular}). Then by Lemma~\ref{lem:timesD},
\[ \sop(DV^{-1})=\frac{\langle a_1,\ldots,a_k\rangle}{d}.\]
Since $U^{-1}$ is a unimodular matrix corresponding to row operations that don't alter the first row, $M=U^{-1}DV^{-1}$ also has
\[\sop(M)=\frac{\langle a_1,\ldots,a_k\rangle}{d},\]
as desired.
\end{proof}

\section{Main Proof}
\label{sec:converse}

We now fill in the remaining holes from the outline in Section~\ref{sec:outline} to give a complete proof of Theorem \ref{thm:main}. We already proved in Proposition~\ref{prop:easy} that all $k$-quotients are $k$-ray-normalescent, so it remains to prove the converse.

Suppose $\nsg$ is $k$-ray-normalescent, that is, there exists a full rank matrix $M\in\ZZ^{k\times k}$ such that $\nsg=\sop(M)$. We want to prove that $\nsg$ is a $k$-quotient. Using Theorem~\ref{thm:SNF} and Lemma~\ref{lem:adj}, it suffices to achieve the following.  

\begin{goal}
\label{goal}
To show that $\nsg$ is a $k$-quotient, it suffices to find a nonsingular matrix $B \in \ZZ^{k \times k}$ with the following properties:
\begin{enumerate}[(a)]
\item 
$B$ has SNF $(1,\ldots,1,d)$,

\item 
$A=\adj(B)$ has first row relatively prime,

\item $\sop(A)=\nsg$.

\end{enumerate}
\end{goal}

\begin{prop} \label{prop:twoproperties}
Let $B \in \ZZ^{k \times k}$ be a nonsingular matrix and let $B'$ denote $B$ with the first column removed. Then $B$ satisfies properties (a) and (b) of Goal~\ref{goal} if and only if the columns of $B'$ form a \emph{primitive set}; that is, if they form a basis for the lattice of integer points contained in their real linear span. 
\end{prop}

\begin{proof} For $1\le i,j\le k$, let $B_{ij}$ denote the minor obtained by removing the $i$th row and $j$th column of $B$.
Property (a) above is equivalent to $d_{k-1}$ (the $(k-1)$-st diagonal element of the SNF) being 1, because $d_i$ divides $d_{i+1}$, for all $i$, and so $d_{k-1}=1$ forces $d_j=1$ for all $j\le k-1$. This is equivalent to the $(k-1)\times (k-1)$ minors of $B$ being relatively prime, i.e.,
\[\gcd\left(B_{ij}:\ 1\le i,j\le k\right)=1.\]

For Property (b), notice that the $i$th entry of the first row of $A$ is $(-1)^{i+1}B_{i1}$, using the standard definition of the adjugate matrix. Therefore, Property (b) is equivalent~to
\[\gcd\left(B_{i1}:\ 1\le i\le k\right)=1.\]

In other words, Property (b) subsumes Property (a), and we are simply looking for $B$ such that $\gcd\left(B_{i1}:\ 1\le i\le k\right)=1.$ Notice that all of these minors remove the first column of $B$, so they are the maximal minors of $B'$.  That is, we need to find $B$ such that the maximal minors of $B'$ are relatively prime. By~\cite[\S 1.3]{lekkerkerker}, this is equivalent to the columns of $B'$ forming a primitive set.
\end{proof}

Now, it is already known~\cite{EW} that the columns of a ``random'' integer matrix $B'$ with more rows than columns will indeed form a primitive set with positive probability.  The precise result is as follows.  

\begin{thm}[{\cite[Theorem 1]{EW}}]\label{thm:primitiveprobability}
Fix $k' < k \in \ZZ_+$. For $q \in \ZZ_+$, $1 \leq i \leq k'$, and $1 \leq j \leq k$, let $b_{q,i,j} \in \ZZ$.  For a given $q$, choose integers $s_{ij}$ uniformly and independently from the set $b_{q,i,j} \leq s_{ij} \leq b_{q,i,j}+q$. Let $\vec s_i = (s_{i1}, \dots, s_{ik})$ and let $S = \{\vec s_1, \vec s_2, \dots, \vec s_k'\}$.
If each $b_{q,i,j}$ is bounded by a polynomial in $q$, then as $q \to \infty$, the probability that $S$ is a primitive set approaches
\[ \frac{1}{\zeta(k)\zeta(k-1)\dots\zeta(d-k+1)} \]
where $\zeta$ is the Riemann zeta function. 
\end{thm}

Thus, Goal~\ref{goal} reduces to finding large regions of $\RR^{k \times k}$ in which the integer matrices $B$ all have Property (c), that is, $\sop(\adj(B)) = \nsg$ (see Lemma~\ref{prop:goodcube}). In such large regions, we will surely be able to find a matrix satisfying (b) (and hence (a)), using Theorem~\ref{thm:primitiveprobability}. This will give us a matrix satisfying Properties (a), (b), and (c) of Goal~\ref{goal}, and we will have found our $k$-quotient.

Recall that $\nsg=\sop(M)$, let $\vec v_1,\ldots,\vec v_k$ be the columns of $M$, and let 
\[\cC=\cone(M)=\cone(\vec v_1,\ldots,\vec v_k).\]
Let $\vec t = \vec v_1 + \cdots + \vec v_k$, which is an integer vector in the (topological) interior of $\cC$. For each positive integer $r$, let $M_r$ be the matrix whose columns are the integer vectors $r \vec v_1 - \vec t, \dots, r \vec v_k - \vec t$, and let $\cC_r=\cone(M_r)$. We will show in Lemma~\ref{lem:semigroupstability2} that for sufficiently large $r$, slight perturbations of $M_r$ still have $\sop(M_r)=\sop(M)$. It is convenient to measure perturbations coordinate-wise, so we will use the element-wise $\ell_\infty$-norm on matrices, that is,
\[\norm{(m_{ij})}=\max_{i,j}\abs{m_{ij}}.\]

\begin{lemma} \label{lem:conechain}
For every integer $r > k$, 
\begin{enumerate}[(a)]
\item 
the extreme rays of $\cC$ are contained in the interior of $\cC_r$, and  

\item 
the extreme rays of $\cC_{r+1}$ are contained in the interior of $\cC_r$.
\end{enumerate}
\end{lemma}

\begin{proof}
First, since $\sum_{i=1}^k \left( r \vec v_i - \vec t \right) = (r-k) \vec t$ belongs to the interior of $\cC_r$ for all $r > k$, we have $\vec t \in \interior(\cC_r)$ and thus $\vec v_i = (\vec v_i - \vec t) + \vec t \in \interior(\cC_r)$ for each $i = 1, \dots, k$, 

Furthermore, for each $i$ we have $(r+1) \vec v_i - t = \left( r\vec v_i - t \right) + \vec v_i$. Since the first term is a generator of $\cC_r$ and the second belongs to $\cC \setminus \{0\}$ which, by the first statement, is contained in the interior of $\cC_r$, we conclude that $(r+1) \vec v_i - t \in \interior(\cC_r)$.
\end{proof}

\begin{lemma} \label{lem:coneperturbation}
Let $r > k$ and let $\vec w_i = (r+1) \vec v_i - \vec t$, the extreme rays of $\cC_{r+1}$. For all sufficiently small $\epsilon > 0$, if $\cC'$ is a cone generated by (real) vectors $\vec w'_1, \dots, \vec w'_k$ such that $\lVert \vec w'_i - \vec w_i \rVert < \epsilon$ for each $i$, then we have $\cC \subseteq \cC' \subseteq \cC_r.$
\end{lemma}

\begin{proof}
By Lemma~\ref{lem:conechain}, for each $i=1,\dots,k$ there exists $\epsilon_i$ such that the ball $B_{\epsilon_i}(\vec w_i)$ is contained in $\cC_r$. So the containment $\cC' \subseteq \cC_r$ will hold whenever $\epsilon < \min\{\epsilon_1, \dots, \epsilon_k\}$.

For the other containment, again by Lemma \ref{lem:conechain} we have $\cC \subseteq \interior(\cC_{r+1})$ So let $\eta_i = \dist(\vec v_i, \partial \cC_{r+1}) > 0$ for each $i$. For the same reason, for each $i$ we can write $\vec v_i = \sum_{j=1}^k \mu_{ij} \vec w_j$ where each coefficient $\mu_{ij}$ is strictly positive. Let $\delta_i = \sum_{j=1}^k \mu_{ij}$. Let
\[ \vec v'_i = \sum_{j=1}^k \mu_{ij} \vec w'_j \in \interior(\cC').\]
Then
\[ \norm{ \vec v'_i - \vec v_i } = \norm{ \sum_{i=1}^k \mu_{ij} \left( \vec w'_i -
\vec w_i \right) }\leq \sum_{j=1}^k \mu_{ij} \epsilon =  \delta_i \epsilon.\]
Thus if $\epsilon < \min \{ \frac{\eta_1}{\delta_1}, \dots, \frac{\eta_k}{\delta_k} \}$, then $\vec v_i \in \cC'$ for each $i$, and thus $\cC \subseteq \cC'$.
\end{proof}

\begin{lemma} \label{lem:semigroupstability}
For all sufficiently large $r$, $\sop(M_r) = \nsg$.
\end{lemma}

\begin{proof}
By Lemma \ref{lem:conechain}, for every $r > k$ we have $\cC \subseteq \cC_r$, so the containment $\nsg \subseteq \sop(M_r)$ is immediate.

For the opposite containment, we will consider the gaps of the semigroup $\nsg$. For each gap $z$, let $H_z$ be the hyperplane $x_1 = z$ in $\R^n$. By definition $H_z \cap \cC$ contains no lattice points. Since $\ZZ^n$ is a closed set and $H_z \cap \cC$ is compact, $u_z := \dist(H_z \cap \cC, \Z^n) > 0$. 

Let $\vec y$ be any (real) point in $\cC_r \cap H_z$. Then we can write 
\[ \vec y = \sum_{i=1}^k \lambda_i \left( r\vec v_i - \vec t \right), \; \lambda_1, \dots, \lambda_k \geq 0.\]
Let $s_i, t'$ be the first coordinates of $\vec v_i,\vec t$, respectively, and note that $z$ is the first coordinate of $\vec y$, so that
\[ z = \sum_{i=1}^k \lambda_i \left(s_i - t' \right) \geq \left(rs_1 - t'\right) \sum_{i=1}^k \lambda_i.\]
On the other hand, the point $\vec w := r \sum_{i=1}^k \lambda_i \vec v_i$ belongs to $\cC$, and we have
\[ \norm{\vec w - \vec y} = \norm{\sum_{i=1}^k \lambda_i \vec t} = \left( \sum_{i=1}^k \lambda_i \right) \norm{\vec t} \leq \frac{z \norm{\vec t}}{rs_1 - t'}.\]
The last expression tends to zero as $r$ grows. Since $\nsg$ has only finitely many gaps and again since $H_z \cap \cC$ is compact, for sufficiently large $r$ it will be the case for every gap $z$ and every $\vec y \in \cC_r \cap H_z$ that
\[ \dist(\vec y, \cC) \le \dist(\vec y, H_z \cap \cC) < u_z.\]
By the definition of $u_z$, it follows that for sufficiently large $r$, any such $\vec y$ cannot be an integer point. Thus $\sop(M_r)$ does not contain any of the gaps, so $\sop(M_r) \subseteq \nsg$.
\end{proof}

\begin{lemma} \label{lem:semigroupstability2}
For all sufficiently small $\varepsilon$ and all sufficiently large natural numbers $r$, if $A$ is a matrix such that $\lVert A - M_{r+1} \rVert < \varepsilon$, then $\sop(A) = \nsg$. 
\end{lemma}

\begin{proof}
Let $\cC'$ be the cone generated by the columns of $A$. By Lemma \ref{lem:coneperturbation}, we have $\cC \subseteq \cC' \subseteq \cC_r$. Then by Lemma \ref{lem:semigroupstability} we have $\nsg = \sop(M) \subseteq \sop(A) \subseteq \sop(M_r) = \nsg.$ 
\end{proof}

The above yields a large region from which we may choose $A$ such that $\sop(A)=\nsg$. Looking at Goal~\ref{goal}, we want to instead choose some $B$ from its own large region, and then use $A=\adj(B)$. The following lemma allows us to do this. 

\begin{prop} \label{prop:goodcube}
There exists a matrix $B_0$ such that, for every positive integer $q$, there exists a cube $\Lambda_q \subseteq \R^{k \times k}$ centered on $qB_0$ and of diameter $q$ such that every matrix $B \in \Lambda_q \cap \Z^{k \times k}$ satisfies $\sop(\adj(B)) = \nsg$. 
\end{prop}

\begin{proof}
Choose $r$ and $\varepsilon$ to satisfy Lemma~\ref{lem:semigroupstability2}, and let $M^{-1}_{r+1} = (m_{ij})$.  By the continuity of the matrix inverse away from singular matrices, there exists $\delta>0$ such that, for all $B\in \R^{k\times k}$ with $\norm{B-M_{r+1}^{-1}}<\delta$, we have that $B$ is nonsingular and $\norm{B^{-1}-M_{r+1}}<\varepsilon$. That is, since we are using the element-wise $\infty$-norm, the conclusion holds whenever $B = (b_{ij})$ such that $\abs{b_{ij}-m_{ij}} < \delta$. 

Let $B_0=\frac{1}{2\delta}M_{r+1}^{-1}$. Suppose $B\in \Z^{k\times k}$ is an integer matrix with 
$\norm{B-qB_0}< q/2$ (that is, $B\in\Lambda_q\cap\Z^{k\times k}$, as in the statement of this proposition). Let $w=q/2\delta$ and note $\left(2\delta B/q\right)^{-1}=wB^{-1}$. Then
\begin{align*}
\norm{B-qB_0}< q/2
&\Rightarrow \norm{B/w-qB_0/w}<q/2w\\ 
& \Rightarrow \norm{B/w-M^{-1}_{r+1}}<\delta\\
&\Rightarrow \norm{wB^{-1}-M_{r+1}}<\varepsilon,
\end{align*} 
and so $\sop\left(wB^{-1}\right)=S$ by Lemma~\ref{lem:semigroupstability2}. Let $A=\adj(B)=\det(B)B^{-1},$ where $\adj(B)$ is the classical adjoint. Scaling a matrix does not change the cone its columns generate,~so
\[\sop\left(A\right)=\sop\left(\frac{\det(B)}{w}wB^{-1}\right)=S,\]
which completes the proof.  
\end{proof}



We are now ready to tie everything together and prove $\nsg=\sop(M)$ is a $k$-quotient. By Proposition~\ref{prop:goodcube}, for every integer $q$, the cube $\Lambda_q$ with center $qB_0$ and diameter $q$ has the property that if $B \in \Lambda_q \cap \ZZ^{k \times k}$, then $\sop(\adj(B)) = \nsg$. That is, all $B\in \Lambda_q \cap \ZZ^{k \times k}$ satisfy Property (c) of Goal~\ref{goal}. Notice that the inequalities defining the cube $\Lambda_q$ depend linearly on $q$. Thus the hypotheses of Theorem~\ref{thm:primitiveprobability} apply to the $k \times (k-1)$ matrix $B'$ obtained by removing the first column of $B$. Therefore, for sufficiently large~$q$, the probability that the columns of $B'$ form a primitive set must be positive, and in particular there exists at least one $B \in \Lambda_q \cap \ZZ^{k \times k}$ such that the columns of $B'$ form a primitive set. By Proposition~\ref{prop:twoproperties}, $B$ will also satisfy Properties (a) and (b) of Goal~\ref{goal}, meaning we have shown that $\nsg$ is a $k$-quotient.

\section*{Acknowledgements}

Tristram Bogart was supported by internal research grant INV-2020-105-2076 from the Faculty of Sciences of the Universidad de los Andes.


\end{document}